\documentclass[reqno,12pt]{amsart}
\usepackage{amsmath,amsthm,amsfonts,amssymb}
\usepackage{euscript}
\usepackage[utf8]{inputenc}
\usepackage[english]{babel}
\usepackage{graphicx}
\usepackage{pgf,tikz}
\usepackage[colorlinks, citecolor=blue]{hyperref}
\usetikzlibrary{arrows}
\date{}
\usepackage{ulem}
\usepackage{epigraph}

\newtheorem{theo}{\indent Theorem}[section]
\newtheorem{lemm}[theo]{\indent Lemma}
\newtheorem{coro}[theo]{\indent Corollary}
\newtheorem{prop}[theo]{\indent Proposition}
\newtheorem{defn}[theo]{\indent Definition}

\newtheorem{quest}[theo]{\indent Question}
\newtheorem{rem}[theo]{\indent Remark}
\renewenvironment{proof}
	{\par\indent{\bf Proof.}} 
	{\hfill$\scriptstyle\blacksquare$}

\makeatletter
  \def\section{\@startsection{section}{2}%
    {\z@}{.5\linespacing\@plus.7\linespacing}{.5em}%
    {\normalfont\bfseries\centering}}
\def\@secnumfont{\bfseries}
\makeatother

\newcommand{\Z}{\mathbb{Z}}
\newcommand{\Q}{\mathbb{Q}}
\newcommand{\X}{\mathbb{X}}
\newcommand{\R}{\mathbb{R}}

\newcommand{\E}{\mathbb{E}}

\newcommand{\SSS}{\mathbb{S}}

\newcommand{\HH}{\mathbb{H}}
\newcommand{\F}{\mathbb{F}}
\newcommand{\tF}{\widetilde{\mathbb{F}}}

\newcommand{\SL}{\mathrm{SL}}

\newcommand{\OO}{\mathrm{O}}
\newcommand{\OOO}{\mathcal{O} }

\newcommand{\id}{\mathrm{Id} }

\newcommand{\Cyc}{\mathrm{Cyc}}

\newcommand{\PO}{\mathbf{PO}}
\newcommand{\tG}{\widetilde{G}}

\DeclareMathOperator*{\Isom}{Isom}

\newcounter{z}

\title{On faces of quasi-arithmetic Coxeter polytopes}

\author{Nikolay Bogachev}
\address{\parbox{\linewidth}{Skolkovo Institute of Science and Technology, Skolkovo, Russia \\
Laboratory of combinatorial and geometric structures, Moscow Institute of Physics \\ and Technology, Dolgoprudny, Russia \\
Caucasus Mathematical Centre, Adyghe State University, Maikop, Russia}}
\email[]{nvbogach@mail.ru}

\author{Alexander Kolpakov}
\address{\parbox{\linewidth}{Institut de Math\'ematiques, Universit\'e de Neuch\^atel, 2000 Neuch\^atel, Suisse/Switzerland \\
Laboratory of combinatorial and geometric structures, Moscow Institute of Physics \\ and Technology, Dolgoprudny, Russia}}
\email[]{kolpakov.alexander@gmail.com}

\begin{document}

\definecolor{uuuuuu}{rgb}{0.26666666666666666,0.26666666666666666,0.26666666666666666}
\definecolor{qqqqff}{rgb}{0.0,0.0,1.0}

\begin{abstract}
    We prove that each lower-dimensional face of a quasi-arithmetic Coxeter polytope, which happens to be itself a Coxeter polytope, is also quasi-arithmetic. We also provide a sufficient condition for a codimension $1$ face to be actually arithmetic, as well as a few computed examples. 
\end{abstract}

\maketitle

\epigraph{In memoriam \`E. B. Vinberg}{1937 -- 2020}

\section{Introduction}\label{section:intro}

Let $\mathbb{X}^n$ be one of the three spaces of constant curvature, i.e. either the Euclidean $n$-space $\E^n$, or the $n$-dimensional sphere $\SSS^n$, or the $n$-dimensional hyperbolic (Lobachevsky) space $\HH^n$. 

Let $P$ be a convex polytope in $\X^n$. The group $\Gamma$ generated by reflections in the supporting hyperplanes of the facets (i.e. codimension $1$ faces) of $P$ is a \textit{discrete reflection group}, if the images of $P$ under the action of $\Gamma$ tessellate $\mathbb{X}^n$ (i.e. $\mathbb{X}^n$ is entirely covered by copies of $P$, such that their interiors do not overlap).  The polytope $P$ in this case is \textit{the fundamental polytope} for $\Gamma$. In particular, $\Gamma$ is discrete whenever any two hyperplanes $H_i$ and $H_j$ bounding $P$ either do not intersect or form a dihedral angle of $\pi/n_{ij}$, where $n_{ij} \in \Z$, $n_{ij} \geq 2$.

If $P$ is compact, then $\Gamma$ is called
\textit{a cocompact reflection group}, and if $P$ has finite volume (in which case $P$ may or may not be compact),
then $\Gamma$ is called \textit{cofinite} or a discrete group of
\textit{finite covolume}.

Discrete reflection groups of finite covolume acting on spheres and Euclidean spaces were classified by Coxeter in 1933 \cite{Cox34}. Therefore, fundamental polytopes of discrete reflection groups in $\X^n$ are called \textit{Coxeter polytopes}. They belong to the class of so-called \textit{acute-angled polytopes}, i.e. those whose dihedral angles are less than or equal to $\pi/2$.

Suppose that $\mathbb{F} \subset \R$ is a totally real number field, and let $\OOO_\F$ denote its ring of integers. Let $G = \mathbf{PO}(n,1)$ be the isometry group of $\HH^n$, and $\widetilde{G}$ be an \textit{admissible} simple algebraic $\F$-group, i.e. $\widetilde{G}(\R) = G$ and $\tG^\sigma(\R)$ is a compact group for any non-identity embedding $\sigma \colon \mathbb{F} \to \mathbb{R}$. Here $H^\sigma$ denotes the algebraic group defined over $\sigma(\F)$ and obtained from an abstract algebraic group $H$ by applying $\sigma$ to the coefficients of all polynomials that define $H$.

Two subgroups $\Gamma_1 $ and $\Gamma_2$ of $G$ are called \textit{commensurable} if, for some element $g \in G$, the group $\Gamma_1 \cap g \Gamma_2 g^{-1}$ is a subgroup of finite index in each of them.

\begin{defn}[\cite{BHC62}]
If $\Gamma \subset G$ is commensurable with  $\widetilde{G}(\OOO_\F)$ then $\Gamma$ is called \textit{arithmetic}. The field $\F$ is called the \textit{ground field} of $\Gamma$.  
\end{defn}

It is known from the general theory, cf. \cite{BHC62} and \cite{MT62}, that if $\Gamma \subset G$ is commensurable with  $\widetilde{G}(\OOO_\F)$ then $\Gamma$ is a 
\textit{lattice} (i.e. a discrete isometry group with a finite volume fundamental polytope). 

\begin{defn}[\cite{Vin67}]
A lattice $\Gamma \subset G = \widetilde{G}(\R)$ is called \textit{quasi-arithmetic} if  $\Gamma \subset \widetilde{G}(\F)$ and \textit{properly quasi-arithmetic} if it is  quasi-arithmetic, but not actually arithmetic. The field $\F$ is called the ground field of $\Gamma$.  
\end{defn}

The history of discrete reflection groups acting on Lobachevsky spaces and, in  particular, arithmetic reflection groups goes back to the 19th century, to the works of Poincare, Fricke, and Klein. A systematic study was started by Vinberg in 1967 \cite{Vin67}. Namely, he developed practically efficient methods that allow to determine compactness or volume finiteness of a given Coxeter polytope according to its Coxeter diagram, and provided a (quasi-)arithmeticity criterion for hyperbolic reflection groups. Later on, Vinberg created an algorithm (colloquially known as ``Vinberg's algorithm'') \cite{Vin72}, that constructs a fundamental Coxeter polytope for any hyperbolic reflection group. Practically, it is most efficient for arithmetic reflection groups associated with Lorentzian lattices (cf. \textbf{\S}~\ref{sec:prel}, \textbf{\S}~\ref{section:t}, \textbf{\S}~\ref{section:l}).

Due to the further results of Vinberg \cite{Vin84}, Long, Maclachlan, Reid \cite{LMR05}, Agol \cite{Ag06}, Nikulin~\cite{Nik81,Nik07}, Agol, Belolipetsky, Storm and  Whyte \cite{ABSW2008}, it became known that  there are only finitely many maximal  arithmetic hyperbolic reflection groups in all dimensions. Moreover, finite volume Coxeter polytopes do not exist in $\HH^{> 995}$ \cite{Hov86, Pro86}, and compact ones do not exist already in $\mathbb{H}^{> 30}$ \cite{Vin84}.

\begin{defn}
Let $P\subset \HH^n$ be a finite volume hyperbolic Coxeter polytope, and $\Gamma = \Gamma(P)$ be the group generated by reflections in the bounding hyperplanes of $P$. Then $P$ is called (quasi-) arithmetic with ground field $\F$ if it is a fundamental domain for a hyperbolic reflection group $\Gamma \subset \Isom(\HH^n)$, and $\Gamma$ is (quasi-)arithmetic with ground field $\F$.  
\end{defn}

By \cite{Vin84} arithmetic Coxeter polytopes with $\F \ne \Q$ in $\mathbb{H}^n$, can exist only for $n < 30$ (because of compactness), and for $14 \le n < 30$ only finitely many ground fields $\F$ are possible.  Non-compact arithmetic Coxeter polytopes (i.e. for $\F = \Q$) exist only for $n \le 21$, $n \ne 20$ \cite{Ess96}.

Recently, Belolipetsky and Thomson
provided infinitely many commensurability classes of properly
quasi-arithmetic hyperbolic lattices in any dimension $n > 2$ (cf. \cite{BT11, Thom16}), as well as Emery \cite{Em17} proved that the covolume of any quasi-arithmetic hyperbolic
lattice is a
rational multiple of the covolume of an arithmetic subgroup.

A natural question is whether a lower-dimensional face of a Coxeter polytope in $\mathbb{H}^n$ is itself a Coxeter polytope in $\mathbb{H}^k$, $1<k<n$. A good indication for this to be often true is the work of Allcock 
\cite[Theorems 2.1 \& 2.2]{Allcock}. Even earlier, Borcherds \cite{Bor87} used a $25$-dimensional infinite volume polytope due to Conway in order to build a $21$-dimensional Coxeter polytope of finite volume as its face.

In the context of arithmetic reflection groups, it is thus natural to ask if arithmeticity is inherited by lower-dimensional Coxeter faces of an arithmetic Coxeter polytope. As our work shows, this is not the case in general.

However, quasi-arithmeticity does descend to Coxeter faces of all dimensions $k>1$.   

\begin{theo}\label{th-face}
Let $P$ be a quasi-arithmetic Coxeter polytope in $\HH^n$ with ground field $\F$ and $P'$ be its $k$-dimensional face, for $2 \le k \le n-1$. If $P'$ is itself a Coxeter polytope, then $P'$ is also quasi-arithmetic with ground field $\F$.
\end{theo}

In the sequel, we enhance the definitions of ground fields and  quasi-arithmeticity to all acute-angled polytopes, cf. Definition \ref{def-quasi-arithm} and the preceding discussion. Then, we have the following easy consequence. 

\begin{coro}\label{cor}
Let $P$ be a Coxeter polytope in $\HH^n$ with ground field $\F = \F(P)$, and let $P'$ be a codimension $k \geq 1$ face of $P$ with ground field $\F' = \F(P')$. If $|\F:\F'| > 1$, then $P$ cannot be quasi-arithmetic. \end{coro}

In \textbf{\S}~\ref{example:prism}, the above property is used as a quick test for polytopes that are not quasi-arithmetic.

The following result shows that arithmetic polytopes can often have arithmetic Coxeter facets. 

\begin{theo}\label{th-facet-arithm}
Let $P$ be an arithmetic Coxeter polytope in $\HH^n$ with ground field $\F$ and $P'$ be its facet (i.e. a codimension $1$ face). Moreover, assume that $P'$ meets all its adjacent facets at dihedral angles of the form $\frac{\pi}{2m}$, for some natural $m\geq 1$, not all necessarily equal to each other.  Then $P'$ is itself an arithmetic Coxeter polytope with ground field $\F$.
\end{theo}

In \textbf{\S}~\ref{section:t}, we give an explicit example shows that once some dihedral angles around $P'$ in the above theorem are not ``even'' (i.e. of the form $\frac{\pi}{2m}$, $m\geq 1$) then $P'$ may be \textit{properly} quasi-arithmetic. 

Another approach would be to descend arithmeticity to reflection centralisers and normalisers in the spirit of Shvartsman's work on \textit{complex} reflection groups \cite{Sch}.

The paper is organised as follows. In \textbf{\S}~\ref{sec:prel} some preliminary facts are given. Then, \textbf{\S}~\ref{sec:th1} is devoted to the proof of Theorem~\ref{th-face} and \textbf{\S}~\ref{sec:th2} is devoted to the proof of Theorem~\ref{th-facet-arithm}.

At the end of the paper, some computed examples are provided (cf. \textbf{\S}~\ref{section:t}--\textbf{\S}~\ref{section:l})
in order to illustrate the above theorems, as well as to enrich the collection of known Coxeter polytopes. All computations were performed by using \texttt{SageMath} computer algebra system \cite{sagemath}.  

The closing remarks (\textbf{\S}~\ref{sec:open}) contain a number of questions that we find quite enticing, although likely hard to approach.

\subsection*{Acknowledgements.} The authors are grateful to \`E.\,B.~Vinberg and D.~Allcock for their encouragement and helpful remarks. The authors would also like to thank the anonymous referees for their comments and suggestions. N.B. is thankful to Institut des Hautes \'Etudes Scientifiques --- IHES, and especially to Fanny Kassel, for their hospitality while this work was carried out.

\subsection*{Funding.} The work of A.K. was supported, in part, by Swiss National Science Foundation [project no. PP00P2-170560] and, in part, by Russian Federation Government [grant no. 075-15-2019-1926]. The work of N.B. on theoretical aspects (\textbf{\S}~\ref{sec:th1} and \textbf{\S}~\ref{sec:th2}) was supported by Russian Federation Government [grant no. 075-15-2019-1926]. The work of N.B. on computational aspects (\textbf{\S}~\ref{section:t} and \textbf{\S}~\ref{section:l}) was partially supported by Russian Science Foundation [project no. 18-71-00153].

\section{Preliminaries}\label{sec:prel}
Let $\mathbb{F} \subset \R$ be a totally real number field, and let $\OOO_\F$ denote its ring of integers.

\begin{defn}
A free finitely generated $\OOO_\F$--module $L$ with an inner product of signature $(n,1)$ is called a $\textit{Lorentzian lattice}$ if, for any non-identity embedding $\sigma \colon \mathbb{F} \to \mathbb{R}$, the quadratic space $L \otimes_{\sigma(\OOO_\F)} \mathbb{R}$ is positive definite (we say that the inner product in $L$ is associated with some admissible Lorentzian quadratic form).
\end{defn}

\begin{rem}
As a matter of terminology, Lorentzian lattices, as defined in the present paper, are also called hyperbolic lattices, cf. \cite{BP18,Bog19,Bog19a,Vin84}.
\end{rem}

Let $L$ be a Lorentzian lattice. Then the vector space $\E^{n, 1} = L \otimes_{\id(\OOO_\F)} \R$ is identified with the $(n + 1)$-dimensional real  \textit{Minkowski space}.
The group $\Gamma = \mathcal{O}'(L)$ of integer (i.e. with coefficients in $\OOO_\F$) linear transformations
preserving the lattice $L$ and mapping each connected
component of the cone 
$$\mathfrak{C} = \{v \in \E^{n,1} \mid (v, v) < 0\} = \mathfrak{C}^+ \cup \mathfrak{C}^-$$
onto itself is a discrete group of motions of the Lobachevsky space $\HH^n$. Here and below we use the \textit{hyperboloid model} for
$\HH^n$, which is the set $\{v \in \E^{n,1} \cap \mathfrak{C}^+  \mid (v, v) = -1\}$.
Its isometry group is $$\Isom(\HH^n) = \PO(n,1) \simeq \OO(L \otimes_{\id(\OOO_\F)} \R)/\{\pm \id\},$$ which is the
group of orthogonal transformations of the Minkowski space $\E^{n, 1}$ that leaves $\mathfrak{C}^+$ invariant.

It is known \cite{BHC62, MT62} that if $\F = \Q$ and the lattice $L$ is isotropic (i.e. the quadratic form  associated with it
represents zero), then the quotient space $\HH^n/\Gamma$ (equivalently, the fundamental domain of $\Gamma$) is non-compact, but of finite volume, and in all other cases it is compact. The case $\F = \Q$ was first studied by Venkov \cite{Ven37}.

\begin{defn}
Any group $\Gamma$ obtained in the way described above, and any subgroup of
$\Isom(\HH^n)$ commensurable to such one, is called an \textit{arithmetic group (or lattice) of simplest type}. The field $\F$ is called the ground field (or the field of definition) of $\Gamma$.
\end{defn}

A primitive vector
$e$ of a Lorentzian lattice $L$ is called a \textit{root} or, more precisely,
a \textit{$k$-root},
where $k = (e, e) \in (\OOO_\F)_{>0}$ if
$2\cdot (e, x) \in k \OOO_\F$ for all $x \in L$. Every root $e$ defines an \textit{orthogonal reflection} (called
a \textit{$k$-reflection} if
$(e, e) = k$) in $L \otimes_{\id(\OOO_\F)} \mathbb{R}$ given by
$$
\mathcal {R}_e: x \mapsto x - 2 \cdot \frac{(e, x)} {(e, e)} \cdot e,
$$
which preserves the lattice $L$. Geometrically speaking, this is a reflection of $\mathbb{H}^n$ with respect to the hyperplane $H_e = \{x \in \mathbb {H} ^ n \mid (x, e) = 0 \},$ called the \textit{mirror} of $\mathcal{R}_e$.

Let $\OOO_r (L)$ denote the subgroup of $\OOO'(L)$ generated by all reflections contained in it.

\begin{defn}
A Lorentzian lattice $L$ is called \textit{reflective} if the index $[\OOO'(L): \OOO_r(L)]$ is finite.
\end{defn}
Clearly, $L$ is reflective if and only if the group $\OOO_r(L)$ has a finite volume fundamental Coxeter polytope.

The results described in \textbf{\S}~\ref{section:intro} give hope that all reflective Lorentzian lattices, as  well as maximal arithmetic hyperbolic reflection groups, can be classified.
For more information about the recent progress on the classification problem cf. 
\cite{Bel16, Bog19, Bog19a}.

\begin{rem}
By a result of Vinberg \cite[Lemma 7]{Vin67}, any arithmetic hyperbolic reflection group is an arithmetic lattice of simplest type with ground field $\F$, and therefore is commensurable with $\OOO_r(L)$, where $L$ is some (necessarily reflective) Lorentzian lattice over a totally real number field~$\F$. 

The same result shows that every 
quasi-arithmetic hyperbolic reflection 
group is contained in a group $\OO(f, 
\F)$, where $f$ is some admissible 
Lorentzian form over $\F$.
\end{rem}

\begin{rem}
As a consequence of Vinberg's algorithm \cite{Vin72} (its recent software implementations are \texttt{AlVin} \cite{Guglielmetti2}, for Lorentzian lattices with an orthogonal basis over several ground fields, and \texttt{VinAl} \cite{VinAlg2017, BP18}, for Lorentzian lattices with an arbitrary basis over $\Q$), reflective lattices provide many explicit examples of arithmetic reflection groups and arithmetic Coxeter polytopes. In contrast, finding properly quasi-arithmetic reflection groups appears to be a harder task. 
\end{rem}

\begin{rem}
The record example of a compact Coxeter polytope was found by Bugaenko
in dimension $8$ \cite{Bug92}, although the maximal possible dimension is bounded
by $30$.

The record example of a finite volume Coxeter polytope is due to Borcherds
in dimension $21$ \cite{Bor87}. It is known that Coxeter polytopes of finite volume
can exist only in dimensions smaller than $996$ \cite{Hov86, Pro86}.

It is worth mentioning that both these examples come from arithmetic reflection groups.
\end{rem}
 
For every hyperplane $H_e$ with unit normal $e$, let us set $H_e^- = \{x \in \E^{n,1} \mid (x,e) \le 0\}$. If 
$$P = \bigcap_{j=0}^N H_{e_j}^-$$ is an acute-angled polytope of finite volume in $\HH^n$, then $G(P) = G(e_0, \ldots, e_N)$ is its Gram matrix and $\tF(P) = \Q(\{g_{ij}\}^N_{i,j=0})$. For a given matrix $A = \{a_{ij}\}^N_{i,j=1}$, let $\Cyc(A)$ denote the set of its cyclic products, i.e. all possible products of the form $a_{i_1 i_2} a_{i_2 i_3} \ldots a_{i_k i_1}$. Let $\F(P) = \Q(\Cyc(G(P))) \subset \tF(P)$. The field $\F(P)$ is called the \textit{ground field} of $P$. 

Given a finite covolume hyperbolic reflection group $\Gamma$, the following criterion allows us to determine if $\Gamma$ is arithmetic, quasi-arithmetic, or neither.

\begin{theo}[Vinberg's arithmeticity criterion \cite{Vin67}]\label{V}
Let $\Gamma$ be a cofinite reflection group acting on $\HH^n$ with fundamental Coxeter polytope $P$. Then $\Gamma$ is arithmetic if and only if
the following conditions hold:
\begin{itemize}
    \item[{\bf(V1)}] $\tF(P)$ is a totally real algebraic number field;
    \item[{\bf(V2)}] for any embedding $\sigma \colon \widetilde\F(P) \to \R$, such that $\sigma\!\mid_{\F(P)} \ne \id$, $G^\sigma(P)$ is positive semi-definite\footnote{Here, a square $(n\times n)$-matrix $M$ is positive semi-definite if $x^t M x \geq 0$ for all $x\in \mathbb{R}^n$.};
    \item[{\bf(V3)}] $\Cyc(2 \cdot G(P)) \subset \OOO_{\F(P)}$.
\end{itemize}
A cofinite reflection group $\Gamma$ acting on $\HH^n$ with fundamental polytope $P$ is  quasi-arithmetic if and only if it satisfies conditions \textbf{(V1)}--\textbf{(V2)}, but not necessarily \textbf{(V3)}.
\end{theo}

Obviously, every arithmetic reflection group is quasi-arithmetic. For most of our arguments, we shall need a wider definition of quasi-arithmetic polytopes that will allow us to capture their main geometric and algebraic properties without paying attention to whether their dihedral angles are of Coxeter type. This will become important in order to transfer from higher-dimensional faces to lower-dimensional ones in a chain of inclusions, which does not necessarily consist entirely of Coxeter polytopes, cf. Proposition~\ref{quasi-facet:acute}. 

\begin{defn}\label{def-quasi-arithm}
Let $P$ be a finite volume acute-angled polytope in $\HH^n$, and let it have well-defined fields $\tF = \tF(P)$,  $\F = \F(P) \subset \tF$, as described above. Then $P$ is called quasi-arithmetic if $\tF$ and $\F$ satisfy conditions \textbf{(V1)}--\textbf{(V2)}.
\end{defn}

After presenting the proof of Theorem~\ref{th-face}, we provide some computed examples of Coxeter polytopes and their faces. For each polytope $P \subset \mathbb{H}^n$, let its \textit{facet tree} $\mathcal{T}(P)$ be a rooted tree with root $P$, where each vertex of level $0 \leq i \leq n-3$ represents the isometry type of a codimension $i$ face $P^{(i)}_j$ of $P$, while the level $i+1$ descendants of $P^{(i)}_j$ represent all distinct isometry types of codimension $1$ faces of $P^{(i)}_j$. Thus, $\mathcal{T}(P)$ shows all possible isometry types of faces of $P$ (in codimensions $1$ through $n-2$), as well as the set of mutually non-isometric facets for each lower-dimensional face. Since $\mathcal{T}(P)$ does not take face adjacency into account, it represents only some geometric and combinatorial features of $P$. Mostly, we shall be interested in determining its Coxeter faces, and classifying them into arithmetic and properly quasi-arithmetic ones. 

\section{Proof of Theorem~\ref{th-face}}\label{sec:th1}

\subsection{Auxiliary results}\label{aux} Below we formulate a few auxiliary lemmas. Their main point is that the properties \textbf{(V1)}--\textbf{(V2)} are inherited by facets of an acute-angled quasi-arithmetic polytope. 

\medskip

In \textbf{\S}~\ref{aux} -- \textbf{\S}~\ref{qa}, $P$ always denotes a finite volume acute-angled polytope in $\HH^n$.

\begin{lemm}\label{lemma:projections}
Let $P' \subset H_{e_0}$ be a facet of $P$,  bounded by the respective hyperplanes $H_{e_1}, \ldots, H_{e_k}$ of $P$ (i.e. $H_{e_0} \cap H_{e_j}$ are the supporting hyperplanes of $P'$ in $\mathbb{H}^{n-1}$). Let $G(P) = \{ g_{ij} \}$ be the Gram matrix of $P$. Then the Gram matrix $G(P') = \{ g'_{ij} \}$ of $P'$ has entries 
$$g'_{ij} = \frac{g_{ij} - g_{0i}\, g_{0j}}{\sqrt{(1-g^2_{0i})\,(1-g^2_{0j})}}.$$
\end{lemm} 
\begin{proof}
Let $e^0_i$, $e^0_j$ be the projections of $e_i$, $e_j$ onto the hyperplane $H_{e_0}$. Then we obtain 
$e^0_j = e_j - (e_j,e_0)e_0$ and
$(e^0_i, e^0_j) = (e_i, e_j) - (e_0, e_i)(e_0, e_j)$. 
Therefore, the lemma follows by setting $e'_i = \frac{e^0_i}{\| e^0_i \|}$, $e'_j = \frac{e^0_j}{\| e^0_j \|}$ to be the respective unit normals of $P'$. 
\end{proof}

\begin{lemm}\label{lemma:Gram}
In the notation of Lemma~\ref{lemma:projections} and its proof, let $G'_s = G(e'_1, \ldots, e'_s)$ be a principal submatrix of $G(P')$, and let $G_s = G(e_0, e_1, \ldots, e_s)$ be the respective principal submatrix of $G(P)$. Then $\det (G'_s)^\sigma \geq 0$ if and only if $\det (G_s)^\sigma \geq 0$, for any Galois embedding $\sigma: \widetilde{\F} \to \R$.
\end{lemm}
\begin{proof}
Let us consider the matrix $G_s$. One can perform the following transformation on its rows: the $j$-th row ($1 \leq j \leq s$) is replaced by the difference of itself and the row corresponding to $e_0$ multiplied by $g_{j 0}$.

After that, each $i$-th column ($1 \leq i \leq s$) of the resulting matrix is divided by $\| e^0_i \| = \sqrt{1 - g^2_{0i}}$, and each $j$-th row ($1 \leq j \leq s$) is divided by $\| e^0_j \| = \sqrt{1 - g^2_{0j}}$. 

According to Lemma~\ref{lemma:projections}, this transformation results in the matrix 
$$
G''_s = 
\begin{pmatrix}
1 & * \\
0 & G'_s(e'_1,\ldots, e'_s)
\end{pmatrix},
$$
and thus we have $\det G_s = \kappa^2 \cdot \det G''_s = \kappa^2 \cdot \det G'_s$, where 
$$
\kappa^{-2} = \prod_{i,j=1}^s \|e^0_i\| \cdot \|e^0_j\| = \prod_{i=1}^s \|e^0_i\|^2.
$$
The above equality holds under all Galois embeddings, and the lemma follows. 
\end{proof}

\begin{lemm}\label{lemma:field-inclusion}
In the notation of Lemma \ref{lemma:projections}, if $P$ is quasi-arithmetic then $\F(P') = \F(P)$. 
\end{lemm}
\begin{proof}
Let $\F = \Q(\Cyc(G(P)))$ and $\F' = \Q(\Cyc(G(P')))$. We shall consider a cyclic product
\begin{equation}\label{eq:cycl}
g'_{i_1 i_2} g'_{i_2 i_3} \ldots g'_{i_s i_1} =
\frac{g_{i_1 i_2} - g_{0 i_1}\, g_{0 i_2}}{\sqrt{(1-g^2_{0 i_1})\,(1-g^2_{0 i_2})}}
\frac{g_{i_2 i_3} - g_{0 i_2}\, g_{0 i_3}}{\sqrt{(1-g^2_{0 i_2})\,(1-g^2_{0 i_3})}}
\ldots
\frac{g_{i_s i_1} - g_{0 i_s}\, g_{0 i_1}}{\sqrt{(1-g^2_{0 i_s})\,(1-g^2_{0 i_1})}}.
\end{equation}
The denominator of the above expression equals
\begin{equation*}
(1-g^2_{0 i_1})(1-g^2_{0 i_2}) \ldots (1-g^2_{0 i_s}) \in \F.
\end{equation*}
Thus, it remains to consider the numerator
\begin{equation}\label{eq:cycl3}
(g_{i_1 i_2} - g_{0 i_1}\, g_{0 i_2})
(g_{i_2 i_3} - g_{0 i_2}\, g_{0 i_3})\ldots
(g_{i_s i_1} - g_{0 i_s}\, g_{0 i_1}).
\end{equation}
While expanding the above expression, we take from each pair of parentheses either $g_{i_m i_{m+1}}$ or $g_{0 i_m } g_{0 i_{m+1}}$, and obtain a sum of cyclic products where each term looks like $g_{i_1 i_2} g_{i_2 i_3} \ldots g_{i_s i_1}$ with some terms of the form $g_{i_m i_{m+1}}$ being replaced by the respective products $g_{0 i_m } g_{0 i_{m+1}}$. 

This is equivalent to replacing some transpositions $(i_m, i_{m+1})$ in $(i_1, i_2)(i_2, i_3) \ldots (i_{s-1}, i_s) = (i_1, i_2, \ldots,$ $i_s)$, with $(0, i_m)(0, i_{m+1})$. Obviously, this operation creates another permutation with a different cycle structure, where either new non-trivial cycles of the elements in $I = \{0, i_1, i_2, \ldots, i_s\}$ will be formed, or some fixed points appear. The latter happens whenever a square of some $g_{ij}$, $i,j \in I$, is present. Thus, each term in \eqref{eq:cycl3} after expansion is a cyclic product from $\Cyc(G(P))$. 

Therefore, a cyclic product of the form (\ref{eq:cycl}) is a  linear combination of cyclic products of $G(P)$ divided by some elements of the field $\F$. This means $g'_{i_1 i_2} g'_{i_2 i_3} \ldots g'_{i_s i_1} \in \F$, and hence $\F' \subset \F$.

Now, let us suppose that $\F' \neq \F$. This implies that there exists a non-identity embedding $\sigma: \F \to \R$ such that $\sigma|_{\F'} = \id$. Then $G^\sigma(P)$ is positive semi-definite, and so is $G^\sigma(P')$, by Lemma~\ref{lemma:Gram}. However, the latter is impossible, since $G^\sigma(P') = G(P')$ is the Gram matrix of a hyperbolic polytope. 
\end{proof}

\subsection{Quasi-arithmeticity of a facet of a quasi-arithmetic acute-angled polytope}\label{qa}

\begin{prop}\label{quasi-facet:acute}
Let $P$ be a quasi-arithmetic acute-angled polytope in $\HH^n$ with ground field $\F = \F(P)$, and let $P'$ be its facet. Then $P'$ is also a quasi-arithmetic acute-angled polytope with the same ground field $\F' = \F(P') = \F$.
\end{prop}
\begin{proof}
It is well-known that a face of an acute-angled polytope is also acute-angled.  Now the proof follows from verifying conditions \textbf{(V1)}--\textbf{(V2)} of Vinberg's arithmeticity criterion (Theorem~\ref{V}).

\textit{Verification of \textbf{(V1)}.}
The field $\widetilde{\F'}$ equals $\widetilde{\F}(\{k^2_{ij}\}^k_{i,j=1})$, where $k^{-1}_{ij} = \|e^0_i\|\cdot \|e^0_j\|$, and thus is a finite extension of $\widetilde{\F}$. It remains to show that $\widetilde{\F'}$ is totally real. Instead, we prove that the larger field $\widetilde{\F''} = \widetilde{\F}(\{\sqrt{1 - g^2_{0i}}\}^k_{i=1})$ is totally real, so that $\widetilde{\F'} \subset \widetilde{\F''}$ is totally real, as well.

To this end, recall that $|g_{0i}| = \cos \angle(H_0, H_i) \le 1$, and thus $1 - g^2_{0i} \geq 0$. Also, $1 - g^2_{0i} = \det G(e_0, e_i)$ is a leading principal minor of $G(P)$ and thus remains positive for any non-identity embedding by Theorem~\ref{V}. Then $\widetilde{\F}(\sqrt{1 - g^2_{0i}})$ is totally real and, consequently, so is $\widetilde{\F''}$.

\textit{Verification of \textbf{(V2)}.} By Lemma~\ref{lemma:Gram} we have that, up to squares in $\F$, all the leading principal minors of $G' = G(P')$ coincide with the corresponding leading principal minors of $G = G(P)$. Since $\F' = \F$ by Lemma~\ref{lemma:field-inclusion}, we have that $\det G^\sigma_s \geq 0$ for every embedding $\sigma: \widetilde{\F} \to \R$, such that $\sigma|_{\F} \ne \id$, by Theorem~\ref{V}, and thus $\det (G')^\sigma_s \geq 0$. \end{proof}

\subsection{Proof of Theorem~\ref{th-face}}

Let $P$ be a quasi-arithmetic Coxeter polytope in $\HH^n$ with ground field $\F = \F(P)$, and let $P'$ be its face of any dimension $\ge 2$ that is itself a Coxeter polytope. Then we need to prove that $P'$ is also a quasi-arithmetic Coxeter polytope  with the same ground field $\F$.

Clearly, the face $P'$ can be included in the following chain of polytopes, by inclusion:
$$
P' = P_1 \subset P_2 \subset \ldots \subset P_t = P,
$$
where $P_j$ is a facet of $P_{j+1}$, for every $1 \leq j \leq t-1$.

Since $P = P_t$ is a quasi-arithmetic Coxeter polytope, then it is an acute-angled one and, by Proposition \ref{quasi-facet:acute}, $P_{t-1}$ is also a quasi-arithmetic acute-angled polytope
with the same ground field. Thus, clearly, each $P_j$ is a quasi-arithmetic acute-angled polytope with 
$\F(P_{j}) = \F(P_{j+1})$ for $1 \le j \le t-1$. \qed

\subsection{A Coxeter prism and its ground field}\label{example:prism}

The $3$-dimensional compact prism $P \subset \mathbb{H}^3$ with Coxeter diagram depicted in Figure~\ref{prism} has one of its bases $P'$ (namely, facet 1) orthogonal to all neighbours (namely, facets 3, 4, and 5). The Coxeter diagram of $P'$ is the sub-diagram in the diagram of $P$ spanned by vertices 3, 4, and 5. 

\begin{figure}[ht]
\centering
\begin{tikzpicture}
\draw[fill=black] 
(5.5,0) circle [radius=.1] node [above] {$1$} node [right] {\hspace{0.3in} $\cosh \ell = \frac{1}{2} \sqrt{5 + 3 \sqrt{2} + 2 \sqrt{5} + \sqrt{10}}$}
(3,0) circle [radius=.1] node [above] {$2$}
(0,-2) circle [radius=.1] node [below] {$3$}
(-3,0) circle [radius=.1] node [above] {$4$}
(0,2) circle [radius=.1] node [above] {$5$}
(3,0) --  (0,-2)
(0,-1.9) -- (-3,0.1)
(0,-2.1) -- (-3,-0.1)
(3,0) --  (0,2)
(0,2.1) -- (-3,0.1)
(0,2) -- (-3,0)
(0,1.9) -- (-3,-0.1)
;
\draw[dashed, very thick] (5.5,0) -- node[above] {$\ell$} (3,0);
\end{tikzpicture}
\caption{A prism in $\mathbb{H}^3$ with ground field $\Q(\sqrt{2}, \sqrt{5})$, one of whose facets has ground field $\Q(\sqrt{5})$}\label{prism}
\end{figure}
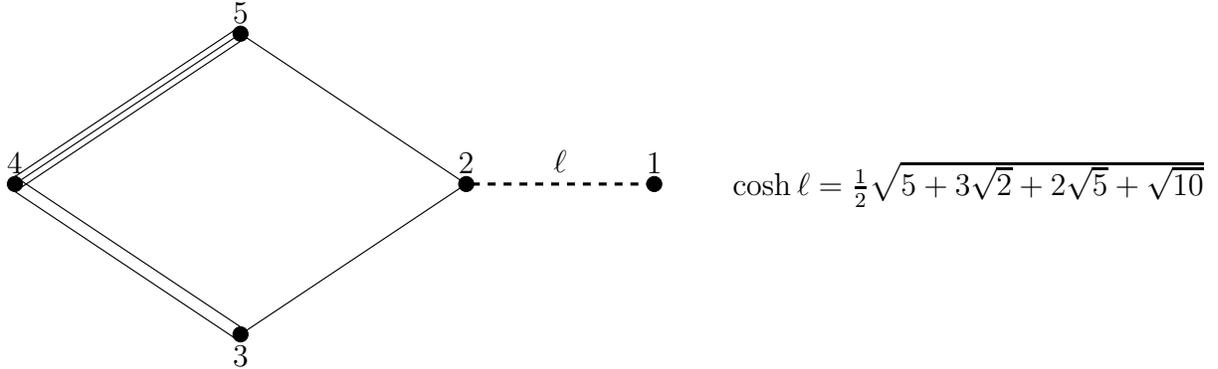

From the Coxeter diagram in Figure~\ref{prism}, one easily gets that $\F(P) = \Q(\sqrt{2}, \sqrt{5})$, while $\F(P') = \Q(\sqrt{5})$. Thus, $P$ cannot be quasi-arithmetic by Corollary~\ref{cor}. Another obstruction to quasi-arithmeticity is the fact that $\cosh^2 \ell$ is not totally positive. However, we do not need to make this computation in order to come up with our conclusion. 

\section{Proof of Theorem~\ref{th-facet-arithm}}
\label{sec:th2}

We start with an auxiliary lemma, that will become useful in the computations below.  

\subsection{Even algebraic integers.} Let $\rho_m = \sin^{-2} \frac{\pi}{2m}$, for $m\geq 2$. We claim that $\rho_m$ is \textit{even}, meaning that $\frac{\rho_m}{2}$ is an algebraic integer. 

\begin{lemm}\label{lemma:even_integer}
For all $m\geq 2$, we have that $\frac{\rho_m}{2}$ is an algebraic integer. 
\end{lemm}
\begin{proof}
Let $p_m(z)$ be the following function:
$$
p_m(z) = \left\{ \begin{array}{cc}
    T_m\left( \frac{z}{\sqrt{2}} \right) \cdot T_m\left( -\frac{z}{\sqrt{2}} \right), & \text{if $m$ is even},  \\
    U_m\left( \frac{z}{\sqrt{2}} \right) \cdot U_m\left( -\frac{z}{\sqrt{2}} \right) + 1, & \text{if $m$ is odd},  \\
\end{array} \right.
$$
where $T_m$, resp. $U_m$, is the $m$-th Chebyshev polynomial of the $1$st, resp. $2$nd, kind. Their basic properties and some relations that will be used below are collected in \cite[Chapter 22]{AbrSte}. 
From the relation $T_{m}(z) = T_{m/2}(2z^2 - 1)$, for even $m$, we see that $T_m\left(\frac{z}{\sqrt{2}}\right)$ is a polynomial. From the recurrence $U_{m+1}(z) = 2z U_m(z) - U_{m-1}(z)$, with $U_0(z)=1$ and $U_1(z)=2z$, we get that $U_m(z)$ is a polynomial in $2z$. This, together with the fact that the odd powers of $\frac{z}{\sqrt{2}}$ in the above expression for $p(z)$ cancel out, means that $p(z)$ is a polynomial with integer coefficients. Moreover, $T_m(z)$ has constant term $\pm 1$ for even $m$, and $U_m(z)$ has constant term $0$ for odd $m$. Thus, $p_m(z)$ has constant term $1$. By using the trigonometric definitions of $T_m(z)$ and $U_m(z)$, we get that $\tau_m = \sqrt{2} \cdot \sin \frac{\pi}{2m}$ is a root of $p_m(z)$. 

Let $\widetilde{p}_m(z) = z^{\deg p_m} \cdot p_m(z^{-1})$ be the reciprocal of $p_m(z)$. Then $\widetilde{p}_m(z)$ is a unitary polynomial having $\tau^{-1}_m = \sqrt{\frac{\rho_m}{2}}$ among its roots. Hence, $\frac{\rho_m}{2}$ is an algebraic integer.
\end{proof}

\subsection{Proof of Theorem~\ref{th-facet-arithm}.}  Let $P'$ be a facet of an arithmetic Coxeter polytope with ground field $\F$. Let $H_{e_0}$ be the supporting hyperplane of $P'$ as a facet of $P$, and let $F_i$, $i=1, \ldots, s$, be the facets of $P'$. Let $H_{e_i}$, $i = 1, \ldots, s$ be the hyperplanes of $P$ such that $H_{e_i} \cap H_{e_0}$ is the supporting hyperplane of $F_i$ in $H_{e_0}$. Since the stabiliser of $H_{e_i} \cap H_{e_0}$ in the reflection group of $P$ is the dihedral group of order $4m$, and $H_{e_0}$ is one of its mirrors, there exists another mirror $H'_i$ (not necessarily a supporting hyperplane for $P$), such that $H'_i$ and $H_{e_0}$ are orthogonal, while $F_i \subset H'_i \cap H_{e_0}$. Thus, the facets of $P'$ come from orthogonal projections of some of the mirrors of the reflection group of $P$ onto the hyperplane $H_{e_0}$. Therefore $P'$ is a Coxeter polytope. 

By Theorem \ref{th-face}, $P'$ is quasi-arithmetic with ground field $\F$, and it remains to verify condition \textbf{(V3)} of Vinberg's arithmeticity criterion. 

Each cyclic product from $\Cyc(2 G(P'))$ has the form $2^s g'_{i_1 i_2} g'_{i_2 i_3} \ldots g'_{i_s i_1}$, which is similar to \eqref{eq:cycl}. Taking into account that the denominator in \eqref{eq:cycl} has each $g_{0i_k} = \cos \frac{\pi}{2m_{i_k}}$, it can be written as  
\begin{equation}\label{eq:cycle4}
2^s \cdot \prod^s_{k=1} \sin^{-2}\left( \frac{\pi}{2m_{i_k}} \right) \cdot (g_{i_1 i_2} - g_{0 i_1}\, g_{0 i_2})
(g_{i_2 i_3} - g_{0 i_2}\, g_{0 i_3})\ldots
(g_{i_s i_1} - g_{0 i_s}\, g_{0 i_1}).
\end{equation}

Same as in \eqref{eq:cycl3}, $g_{ij}$'s form some cyclic products from $\Cyc(G(P))$. Let $t$ be the cardinality of the set $I = \{i_k\, |\, m_{i_k} = 1\}$. Since $\rho_{m_{i_k}} = \sin^{-2} \left( \frac{\pi}{2m_{i_k}} \right) \in 2 \OOO_\F$ by Lemma~\ref{lemma:even_integer} for each $i_k \notin I$, we can rewrite \eqref{eq:cycle4} as
\begin{equation}\label{eq:cycle5}
2^{2 s-t} \cdot \rho \cdot (g_{i_1 i_2} - g_{0 i_1}\, g_{0 i_2})
(g_{i_2 i_3} - g_{0 i_2}\, g_{0 i_3})\ldots
(g_{i_s i_1} - g_{0 i_s}\, g_{0 i_1}),
\end{equation}
with $\rho \in \OOO_\F$.  Let $\mu = 2s - t$. The longest cyclic product that appears in the above expression has length $\lambda = 2s$, if $t=0$, or $\lambda = 2s-t-1$, if $t\geq 1$ (which happens whenever we have $g_{0i_j} = \ldots = g_{0i_{j+t-1}} = 0$ for $t$ consecutive terms). Thus, each term in \eqref{eq:cycle5} of length $\lambda$ is multiplied by $2^{\mu}$ with $\mu \geq \lambda$. Due
to the arithmeticity of $P$, each such product belongs to $\OOO_\F$ by Vinberg's criterion, cf. Theorem~\ref{V}.
 
\section{A polytope of Bugaenko}\label{section:t}

In this section we consider the polytope $P$ with Coxeter diagram depicted in Figure~\ref{polytope-compact}.  This is a compact polytope in $\mathbb{H}^7$ first described by Bugaenko \cite{Bugaenko}, and later on included into a larger census by Felikson and Tumarkin \cite{FT08}. This is an arithmetic polytope: $P$ is the fundamental polytope for $\mathcal{O}_r(L)$, where the lattice $L$ is associated with the quadratic form $q(x) = - \frac{1+\sqrt{5}}{2}\, x^2_0 + x^2_1 + \ldots + x^2_7$.

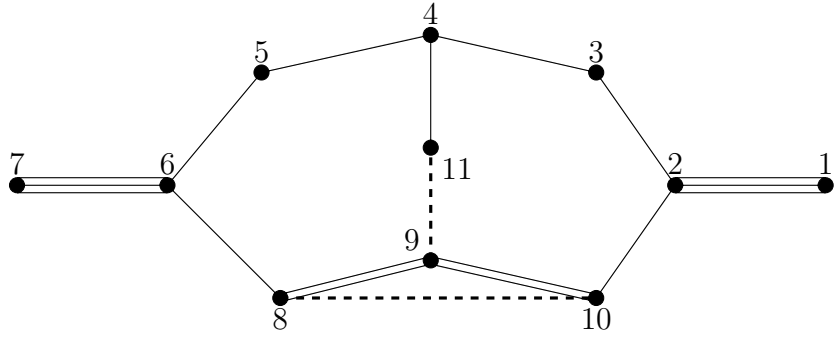
\begin{figure}[ht]
\centering
\begin{tikzpicture}
\draw[fill=black] 
(5*0.75, 5*0.7) circle [radius=.1] node [above] {$2$}
(5*0.54, 5*1.0) circle [radius=.1] node [above] {$3$}
(5*0.1, 5*1.1)  circle [radius=.1] node [above] {$4$}
(-0.35*5, 5*1.0) circle [radius=.1] node [above] {$5$}
(-0.6*5, 5*0.7) circle [radius=.1] node [above] {$6$}
(-0.3*5, 5*0.4) circle [radius=.1] node [below] {$8$}
(5*0.1, 5*0.5) circle [radius=.1] node [above left] {$9$}
(5*1.15, 5*0.7) circle [radius=.1] node [above] {$1$}
(-1*5, 5*0.7) circle [radius=.1] node [above] {$7$}
(5*0.1, 5*0.8) circle [radius=.1] node [below right] {$11$}
(5*0.54, 5*0.4) circle [radius=.1] node [below] {$10$}
(5*0.54, 5*0.4) -- (5*0.75, 5*0.7)
(5*0.75, 5*0.7) -- (5*0.54, 5*1.0)
(5*0.54, 5*1.0) -- (5*0.1, 5*1.1)
(5*0.1, 5*1.1) -- (-0.35*5, 5*1.0)
(-0.35*5, 5*1.0) -- (-0.6*5, 5*0.7)
(-0.6*5, 5*0.7) -- (-0.3*5, 5*0.4)
(5*0.1, 5*1.1) -- (5*0.1, 5*0.8)
(-0.3*5, 5*0.4+0.05) -- (5*0.1, 5*0.5+0.05)
(-0.3*5, 5*0.4-0.05) -- (5*0.1, 5*0.5-0.05)
(5*0.1, 5*0.5+0.05) -- (5*0.54, 5*0.4+0.05)
(5*0.1, 5*0.5-0.05) -- (5*0.54, 5*0.4-0.05)
(-0.6*5, 5*0.7-0.1) -- (-1*5, 5*0.7-0.1)
(-0.6*5, 5*0.7) -- (-1*5, 5*0.7)
(-0.6*5, 5*0.7+0.1) -- (-1*5, 5*0.7+0.1)
(5*0.75, 5*0.7-0.1) -- (5*1.15, 5*0.7-0.1)
(5*0.75, 5*0.7) -- (5*1.15, 5*0.7)
(5*0.75, 5*0.7+0.1) -- (5*1.15, 5*0.7+0.1)
;
\draw[dashed, very thick] (5*0.1, 5*0.5) -- (5*0.1, 5*0.8);
\draw[dashed, very thick] (-0.3*5, 5*0.4) -- (5*0.54, 5*0.4);
\end{tikzpicture}
\caption{A polytope $P$ with ground field $\F(P) = \Q(\sqrt{5})$ from \cite{Bugaenko}}\label{polytope-compact}
\end{figure}

The outer normals to the facets of $P$ can be easily computed by using 
\texttt{AlVin} \cite{Guglielmetti1, Guglielmetti2} and are included in the \texttt{PLoF}\footnote{\textbf{P}olytope's \textbf{Lo}wer-dimensional \textbf{F}aces} worksheet \cite{plof} created for \texttt{SageMath} \cite{sagemath}. 

The geometric characteristics of its lower-dimensional faces can be found by using \texttt{PLoF}. These include Coxeter diagrams and Gram matrices of all isometry types of faces from dimension $7$ down to $2$. 

Furthermore, \texttt{PLoF} builds the facet tree $\mathcal{T}(P)$ for $P$. Here, we would like to stress the fact that on each level of the tree only \textit{isometry types} of faces are given, and thus the adjacency structure of $P$ is not entirely revealed. Another feature is that each isometry type of non-Coxeter faces happens only once in $\mathcal{T}(P)$, since we prune the tree by removing duplicates. However, Coxeter faces are always preserved, so that their inclusion chains can be easily observed.

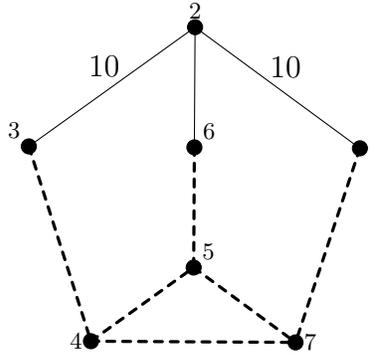
\begin{figure}[ht]
    \centering
    \begin{tikzpicture}[line cap=round,line join=round,>=triangle 45,x=1.0cm,y=1.0cm]
\draw [dashed, very thick] (-0.8*0.6301687854462221,0.8*1.5960492890907991)-- (0.8*1.0770546473922904,0.8*2.821417146526391);
\draw [dashed, very thick] (0.8*1.0770546473922904,0.8*2.821417146526391)-- (0.8*2.770009787007278,0.8*1.5764106682619587);
\draw [dashed, very thick] (-0.8*0.6301687854462221,0.8*1.5960492890907991)-- (0.8*2.770009787007278,0.8*1.5764106682619587);
\draw [dashed, very thick] (0.8*1.0885979390908462,0.8*4.819991966575433)-- (0.8*1.0770546473922904,0.8*2.821417146526391);
\draw [dashed, very thick] (-0.8*1.6622043099335837,0.8*4.835879944572051)-- (-0.8*0.6301687854462221,0.8*1.5960492890907991);
\draw (-0.8*1.6622043099335837,0.8*4.835879944572051)-- (0.8*1.1001412307894027,0.8*6.818566786624475) node [midway, above] {$10\,\,\,$};
\draw (0.8*1.1001412307894027,0.8*6.818566786624475)-- (0.8*1.0885979390908462,0.8*4.819991966575433);
\draw (0.8*1.1001412307894027,0.8*6.818566786624475) -- (0.8*3.839400188115276,0.8*4.804103988578815) node [midway, above] {$\,\,\,10$};
\draw [dashed, very thick] (0.8*3.839400188115276,0.8*4.804103988578815)-- (0.8*2.770009787007278,0.8*1.5764106682619587);
\begin{scriptsize}
\draw [fill=black] (-0.8*0.6301687854462221,0.8*1.5960492890907991) circle (.1)
node [left] {$4$};
\draw  [fill=black](0.8*2.770009787007278,0.8*1.5764106682619587) circle (.1)
node [right] {$7$};
\draw [fill=black] (0.8*3.839400188115276,0.8*4.804103988578815) circle (.1)
node [above right] {$1$};
\draw [fill=black] (0.8*1.1001412307894027,0.8*6.818566786624475) circle (.1)
node [above] {$2$}
;
\draw [fill=black] (-0.8*1.6622043099335837,0.8*4.835879944572051) circle (.1)
node [above left] {$3$};
\draw [fill=black] (0.8*1.0885979390908462,0.8*4.819991966575433) circle (.1)
node [above right] {$6$};
\draw [fill=black]  (0.8*1.0770546473922904,0.8*2.821417146526391) circle (.1) 
node [above right] {$5$};
\end{scriptsize}
\end{tikzpicture}
    \caption{A $3$-dimensional Coxeter face $P'$ of $P$. Here, $\F(P') = \Q(\sqrt{5})$}\label{fig:3dim}
\end{figure}

Let us remark that $P$ has properly quasi-arithmetic faces only in dimension $2$. One of them, the complete list being computed by \texttt{PLoF} \cite{plof}, belongs to a $3$-dimensional Coxeter face $P'$ with Coxeter diagram in Figure~\ref{fig:3dim}.  The subdiagram of $P$ giving rise to $P'$ is generated by the two triple bonds $1 - 2$ and $6 - 7$ in Figure~\ref{polytope-compact}. Then, the diagram of $P'$ can be easily computed by using \texttt{PLoF} \cite{plof}, or by applying the diagrammatic method of \cite[Theorem 2.2]{Allcock}. 

The face $P''_1$ in question has label $2$ in the Coxeter diagram of $P'$. This is a right-angled hexagon with Gram matrix
\begin{equation*}
\resizebox{0.95\linewidth}{!}{$\displaystyle{
G(P^{\prime\prime}_1) = \left( \begin{array}{cccccc}  1 & -2 \sqrt{5}-5 & 0 & 0 & -\sqrt{\frac{1}{3} \left(2 \sqrt{5}+5\right)} & -2 \sqrt{5}-4 \\  
-2 \sqrt{5}-5 & 1 & -2 \sqrt{5}-4 & 0 & -\sqrt{\frac{1}{3} \left(2 \sqrt{5}+5\right)} & 0 \\  
0 & -2 \sqrt{5}-4 & 1 & -\frac{1}{2} \sqrt{\sqrt{5}+5} & 0 & -\frac{1}{2} \left(3 \sqrt{5}+5\right) \\  
0 & 0 & -\frac{1}{2} \sqrt{\sqrt{5}+5} & 1 & -\frac{\sqrt{5}+1}{\sqrt{6}} & -\frac{1}{2} \sqrt{\sqrt{5}+5} \\  -\sqrt{\frac{1}{3} \left(2 \sqrt{5}+5\right)} & -\sqrt{\frac{1}{3} \left(2 \sqrt{5}+5\right)} & 0 & -\frac{\sqrt{5}+1}{\sqrt{6}} & 1 & 0 \\  
-2 \sqrt{5}-4 & 0 & -\frac{1}{2} \left(3 \sqrt{5}+5\right) & -\frac{1}{2} \sqrt{\sqrt{5}+5} & 0 & 1 \end{array} \right).}$}
\end{equation*}

One can easily verify that $P''_1$ is not arithmetic. Let us point out that $P''_1$ makes angles of $\frac{\pi}{2}$, $\frac{\pi}{3}$, and $\frac{\pi}{10}$ with its neighbours. Thus, the ``evenness of the dihedral angles'' condition in Theorem~\ref{th-facet-arithm} seems fairly reasonable.    

Another face $P''_2$ of $P'$ is labelled $1$ in Figure~\ref{fig:3dim}, and makes ``even'' angles of $\frac{\pi}{2}$ and $\frac{\pi}{10}$ with its neighbours. This face is therefore arithmetic by Theorem~\ref{th-facet-arithm}. This is a right-angled pentagon with Gram matrix
\begin{equation*}
\resizebox{0.95\linewidth}{!}{$\displaystyle{
G(P^{\prime\prime}_2) = \left( \begin{array}{ccccc}  1 & -\sqrt{2 \sqrt{5}+5} & 0 & 0 & -\frac{1}{2} \left(\sqrt{5}+1\right) \\  -\sqrt{2 \sqrt{5}+5} & 1 & -\frac{1}{2} \left(\sqrt{5}+3\right) & 0 & 0 \\  0 & -\frac{1}{2} \left(\sqrt{5}+3\right) & 1 & -\frac{1}{2} \sqrt{\sqrt{5}+5} & 0 \\  0 & 0 & -\frac{1}{2} \sqrt{\sqrt{5}+5} & 1 & -\frac{\sqrt{5}+1}{2 \sqrt{2}} \\  -\frac{1}{2} \left(\sqrt{5}+1\right) & 0 & 0 & -\frac{\sqrt{5}+1}{2 \sqrt{2}} & 1 \\ \end{array} \right).}$}
\end{equation*}

\begin{rem}
We also computed the faces of Bugaenko's compact polytope in $\HH^8$, which did not give very informative outcome. Indeed, it happens to have only few Coxeter faces, all of which are arithmetic. The complete computation can be found in \cite{plof}.
\end{rem}

\section{A curious reflective lattice}\label{section:l}

In this example, we consider the reflective lattice $L$ associated with the Lorentzian quadratic form $f(x) = -15\, x^2_0 + x^2_1 + \ldots + x^2_5$. Let $P$ be the fundamental polytope for $\mathcal{O}_r(L)$.  An interesting fact is that $P$ has a descending chain entirely of Coxeter faces starting from the polytope itself that ends up with two $2$-dimensional faces: one arithmetic, and the other properly quasi-arithmetic. None of the previous examples has such a long chain of Coxeter faces. This can be observed by using \texttt{PLoF} \cite{plof}, which also allows visualising all the respective Coxeter diagrams.

Namely, this chain ends in a $2$-dimensional face $P'_1$ with Gram matrix 
$$
G(P'_1) = \left(\begin{array}{rrr}
1 & -\frac{1}{2} \, \sqrt{2} & -1 \\
-\frac{1}{2} \, \sqrt{2} & 1 & 0 \\
-1 & 0 & 1
\end{array}\right),
$$
and another $2$-dimensional face $P'_2$ with Gram matrix 
$$
G(P'_2) = \left(\begin{array}{rrrr}
1 & -\frac{1}{2} \, \sqrt{2} & -\frac{2}{3} \, \sqrt{3} & 0 \\
-\frac{1}{2} \, \sqrt{2} & 1 & 0 & -\frac{1}{2} \, \sqrt{10} \\
-\frac{2}{3} \, \sqrt{3} & 0 & 1 & 0 \\
0 & -\frac{1}{2} \, \sqrt{10} & 0 & 1
\end{array}\right).
$$
It is easy to check that $P'_1$ is arithmetic, while $P'_2$ is properly quasi-arithmetic. 

Moreover, since $15$ is not a sum of three rational squares, $P$ has a descending chain of faces corresponding to the restrictions of $f$ onto the subspaces $x_5 = \ldots = x_{5-i} = 0$, for $i=0,1,2,3$, that starts with two non-compact finite volume faces and ends with two compact ones. All of them are obviously arithmetic (cf.~\cite[Theorem 2.1]{Bug92}). 

\section{Some open questions}\label{sec:open}

Below we list some, to the best of our knowledge, open problems, that seems interesting to us and are related to the above discussion of (quasi-)arithmeticity.

\begin{quest}\label{q1}
Do there exist compact Coxeter polytopes in $\HH^{\ge 9}$?
\end{quest}

Let us recall that the record example is due to Bugaenko in $\HH^8$, and there have been no compact polytopes found in higher dimensions since almost 30 years to date.

\begin{quest}\label{q2}
Do there exist non-arithmetic (in particular, properly quasi-arithmetic) Coxeter polytopes in $\HH^{\ge 6}$, compact or non-compact?
\end{quest}

Some related work was done by Vinberg in the non-compact case \cite{Vin14} for all dimensions $n \leq 12$ and $n = 14, 18$ by constructing analogues of the non-arithmetic hybrids due to Gromov and Pyatetski--Shapiro (later on, Thomson~\cite{Thom16} proved that these are never quasi-arithmetic). There are also some examples among Coxeter prisms, including properly quasi-arithmetic ones \cite{Vin67}. If more Coxeter polytopes become available, one can use the same idea and try to fill in the gaps for $n = 13, 15, 16, 17$, as well as to answer Question~\ref{q2} in full generality.

\begin{quest}\label{q3}
Is it true that a quasi-arithmetic Coxeter polytope in $\HH^{\ge 4}$ with ground field $\Q$ cannot be compact?
\end{quest} 

This question is motivated by the examples of cocompact quasi-arithmetic subgroups of $\SL_2(\Q)$ constructed by Vinberg in \cite{Vin-Bielefeld} that preserve an isotropic quadratic form.

If Question~\ref{q3} has affirmative answer, then Theorem \ref{th-face} limits us only to non-compact polytopes appearing as Coxeter $k$-dimensional faces, for $k \ge 4$, of arithmetic polytopes with ground field $\Q$.

\begin{quest}\label{q4}
Does there exist an arithmetic Coxeter polytope in $\HH^{\ge 4}$ that has a properly quasi-arithmetic Coxeter face of small codimension?
\end{quest}

The fact that the only properly quasi-arithmetic faces we found in Coxeter polytopes from \textbf{\S}~\ref{section:t} -- \textbf{\S}~\ref{section:l} have dimension $2$ is the main motivation for Question~\ref{q4}.  

\begin{quest}\label{q5}
Are there only finitely many maximal quasi-arithmetic hyperbolic reflection groups in all dimensions? Is it true at least for the fixed dimension and degree $d = [\F : \Q]$ of the ground field $\F$?
\end{quest}

This question is naturally motivated by the 
affirmative answer for arithmetic groups. It is 
also interesting, whether Nikulin's methods (cf. 
\cite{Nik81,Nik07}) or the spectral method
(cf.~\cite{ABSW2008,Bel16}).
can be applied.

\end{document}